\documentclass{amsart}

\usepackage[all]{xy}
\usepackage{graphicx}
\usepackage{amsmath}
\usepackage{amsfonts}

\newtheorem{theorem}{Theorem}[section]

\newtheorem{corollary}[theorem]{Corollary}
\newtheorem{lemma}[theorem]{Lemma}
\newtheorem{proposition}[theorem]{Proposition}
\theoremstyle{definition}
\newtheorem{definition}[theorem]{Definition}
\theoremstyle{remark} \theoremstyle{remark}
\newtheorem{remark}[theorem]{Remark}
\newtheorem{example}[theorem]{Example}

\numberwithin{equation}{section}
\newcommand{\Rr}{\mathbb R}
\newcommand{\bS}{\mathbb S}

\renewcommand{\d}{\mathrm d}                              

\title{Cosymplectic $p$-spheres}

\author[B. Cappelletti-Montano]{Beniamino Cappelletti-Montano}
 \address{Dipartimento di Matematica e Informatica, Universit\`a degli Studi di
 Cagliari, Via Ospedale 72, 09124 Cagliari, Italy}
 \email{b.cappellettimontano@gmail.com}

\author[A. De Nicola]{Antonio De Nicola}
 \address{CMUC, Department of Mathematics, University of Coimbra, 3001-501 Coimbra, Portugal}
 \email{antondenicola@gmail.com}

\author[I. Yudin]{Ivan Yudin}
 \address{CMUC, Department of Mathematics, University of Coimbra, 3001-501 Coimbra, Portugal}
 \email{yudin@mat.uc.pt}


\thanks{Research partially supported by CMUC, funded by the European program
COMPETE/FEDER, by FCT (Portugal) grants PEst-C/MAT/UI0324/2011 (A.D.N. and I.Y.), by MICINN (Spain) grants
MTM2011-15725-E, MTM2012-34478 (A.D.N.),  and by PRIN 2010/11 -- Variet\`{a} reali e complesse: geometria, topologia e analisi armonica -- Italy (B.C.M.).}

\begin{document}

\begin{abstract}
We introduce cosymplectic circles and cosymplectic spheres, which are the analogues in the cosymplectic setting of contact circles and contact spheres. We provide a complete classification of compact 3-manifolds that admit a cosymplectic circle.
The properties of tautness and roundness for a cosymplectic $p$-sphere are studied. To any taut cosymplectic circle on a three-dimensional manifold $M$ we are able to canonically associate  a complex structure and a conformal symplectic couple on $M\times \Rr$.
We prove that a cosymplectic circle in dimension three is round if and only if it is taut. On the other hand, we provide examples in higher dimensions of cosymplectic circles which are taut but not round and examples of cosymplectic circles which are round but not taut.
\end{abstract}

\maketitle

\section{Introduction}
The notion of cosymplectic structure was introduced by Libermann in the late 50s as a pair ($\eta$, $\Omega$), where $\eta$ is a closed 1-form and $\Omega$ is a closed 2-form on a $(2n+1)$-dimensional manifold $M$, such that $\eta\wedge\Omega^n$ is a volume form.
Cosymplectic manifolds play an important role in the geometric description of time-dependent mechanics (see \cite{survey} and references therein). Starting from 1967, when Blair defined an adapted Riemannian structure on a cosymplectic manifold, a study of the metric properties on these manifolds has been also carried forward. Recently, new remarkable results on cosymplectic manifolds and their Riemannian counterpart (sometimes called
coK{\"{a}}hler manifolds) appeared in \cite{bazzoni1, fetcu, li}.

In this paper,  we introduce the notion of cosymplectic circle and, more generally, of cosymplectic $p$-sphere. We were inspired by the concepts  of contact circle and contact $p$-sphere, introduced by Geiges and Gonzalo (\cite{geiges95}) and then generalized by Zessin (\cite{zessin05}). We brought these ideas into the cosymplectic setting.

Let $(\eta_{1},\Omega_{1})$ and $(\eta_{2},\Omega_{2})$ be two cosymplectic structures on a manifold $M$. We say that they generate a cosymplectic circle if the pair  $(\lambda_1\eta_{1}+\lambda_2\eta_{2},\lambda_1\Omega_{1}+\lambda_2\Omega_{2})$ is a  cosymplectic structure for \ every \ $\lambda=(\lambda_{1}, \lambda_{2})\in\mathbb{S}^{1}$. In \cite{geiges95} Geiges and Gonzalo classified closed 3-manifolds that admit a taut contact circle.
In the present paper, we provide a complete classification of 3-dimensional compact manifolds that admit a cosymplectic circle.
Furthermore, we introduce the notion of tautness and of roundness for a cosymplectic $p$-sphere. To any taut cosymplectic circle on a three-dimensional manifold $M$ we canonically associate a complex structure and a conformal symplectic couple on $M\times \Rr$.
In dimension three a cosymplectic circle is proven to be round if and only if it is taut. In higher dimensions we provide examples of cosymplectic circles which are taut but not round and examples of cosymplectic circles which are round but not taut.

It is not difficult to show that there exist no cosymplectic circles in any manifold of dimension $4n+1$ and, a fortiori, no cosymplectic $p$-spheres.
In the last section of the paper, devoted to the relation between Riemannian 3-structures and almost cosymplectic spheres, we show that 3-cosymplectic manifolds (\cite{bai1}) provide a source of examples of cosymplectic spheres which are both round and taut in any dimension $4n+3$.
We also improve the result of Zessin  that any 3-Sasakian manifold admits a contact sphere which is both taut and round (cf. \cite[Proposition 4]{zessin05}). Indeed, we show that any 3-Sasakian manifold admits a sphere of Sasakian structures which is both taut and round.
We do this by introducing  the class of quasi-contact spheres, which include both cosymplectic (rank 1) and contact spheres (rank 2n+1), and we prove that any 3-quasi-Sasakian manifold defines a taut and round quasi-contact sphere.

\section*{Acknowledgements}
The authors are grateful to the anonymous reviewer for a careful checking of the details and for correcting the proof of Corollary 4.4.


\section{Main definitions and examples}
An \emph{almost cosymplectic structure} on a smooth manifold $M$ is given by a pair $(\eta,\Omega)$, where $\eta$ is a $1$-form  and $\Omega$ a  $2$-form on $M$ such that $\eta\wedge\Omega^n$ is a volume form. Thus, in particular, $\dim(M)=2n+1$. The almost cosymplectic structure $(\eta,\Omega)$ is said to be a \emph{contact
 structure} if $d\eta=\Omega$ and a \emph{cosymplectic structure} if $\eta$ and $\Omega$ are both closed.

On any almost cosymplectic manifold there exists a global vector field $\xi$ uniquely determined by the conditions
\begin{equation}\label{reeb}
i_{\xi}\eta=1, \ \ \ i_{\xi}\Omega=0,
\end{equation}
and called the \emph{Reeb vector field}.

\begin{definition}\label{acsphere}
Let $(\eta_{1},\Omega_{1}), \ldots, (\eta_{p+1},\Omega_{p+1})$ be $p+1$ almost cosymplectic structures on $M$. Consider the family
$\left\{(\eta_{\lambda},\Omega_{\lambda})\right\}_{\lambda\in\mathbb{S}^p}$ where
\begin{equation*}
\eta_{\lambda}:=\lambda_{1}\eta_{1}+\ldots+\lambda_{p+1}\eta_{p+1}, \ \ \ \Omega_{\lambda}:=\lambda_{1}\Omega_{1}+\ldots+\lambda_{p+1}\Omega_{p+1}.
\end{equation*}
If the pair  $(\eta_{\lambda},\Omega_{\lambda})$ is an almost cosymplectic structure for  every $\lambda=(\lambda_{1}, \ldots, \lambda_{p+1})\in\mathbb{S}^{p}$, then the family
$\left\{(\eta_{\lambda},\Omega_{\lambda})\right\}_{\lambda\in\mathbb{S}^p}$  is called  an \emph{almost  cosymplectic  $p$-sphere  generated  by  $(\eta_{1},\Omega_{1}),  \ldots, (\eta_{p+1},\Omega_{p+1})$}. An almost
cosymplectic $p$-sphere  is called  an \emph{almost cosymplectic circle} or an \emph{almost cosymplectic sphere} if $p=2$ or $p=3$, respectively.
\end{definition}

We shall denote an almost cosymplectic $p$-sphere either by $\left\{(\eta_{\lambda},\Omega_{\lambda})\right\}_{\lambda\in\mathbb{S}^p}$ or by indicating a set of generators, such as $\left\{(\eta_{1},\Omega_{1}), \ldots, (\eta_{p+1},\Omega_{p+1})\right\}$.

There are two interesting types of almost cosymplectic $p$-spheres, which are introduced by the following definitions.

\begin{definition}
An almost cosymplectic $p$-sphere $\left\{(\eta_{\lambda},\Omega_{\lambda})\right\}_{\lambda\in\mathbb{S}^p}$   is said to be \emph{taut} if all its
elements give the same volume form, i.e.
\begin{equation*}
\left(\sum_{i=1}^{p+1}\lambda_{i}\eta_{i}\right)  \wedge \left(\sum_{j=1}^{p+1}\lambda_{j}\Omega_{j}\right)^n = \left(\sum_{i=1}^{p+1}\lambda'_{i}\eta_{i}\right)  \wedge \left(\sum_{j=1}^{p+1}\lambda'_{j}\Omega_{j}\right)^n\end{equation*}
for any $(\lambda_{1}, \ldots, \lambda_{p+1}),(\lambda'_{1}, \ldots, \lambda'_{p+1})\in \mathbb{S}^{p}$.
\end{definition}

\begin{definition}
An almost cosymplectic $p$-sphere $\left\{(\eta_{\lambda},\Omega_{\lambda})\right\}_{\lambda\in\mathbb{S}^p}$ is said to be \emph{round} if for any
$\lambda=(\lambda_{1}, \ldots, \lambda_{p+1})\in \mathbb{S}^{p}$ the vector field $\lambda_{1}\xi_{1}+\ldots+\lambda_{p+1}\xi_{p+1}$ is the Reeb vector field of the almost cosymplectic structure $(\eta_{\lambda},\Omega_{\lambda})$.
\end{definition}

Let $\left\{(\eta_{\lambda},\Omega_{\lambda})\right\}_{\lambda\in\mathbb{S}^p}$ be an almost cosymplectic $p$-sphere. Notice that if all the almost
cosymplectic structures $(\eta_{1},\Omega_{1}), \ldots, (\eta_{p+1},\Omega_{p+1})$ are contact structures, then any element
$(\eta_{\lambda},\Omega_{\lambda})$ of the almost cosymplectic $p$-sphere  is a contact structure, since
$d\eta_{\lambda}=\lambda_{1} d\eta_{1}+\ldots+\lambda_{p}d\eta_{p}=\lambda_{1}\Omega_{1}+\ldots+\lambda_{p}\Omega_{p}=\Omega_{\lambda}$. Thus  Definition
\ref{acsphere} generalizes the notion of \emph{contact circles} and \emph{contact spheres}, introduced by Geiges and Gonzalo in \cite{geiges95}, and that
one of \emph{contact $p$-spheres}, introduced by Zessin in \cite{zessin05}.

On the other hand, if the generating almost cosymplectic structures $$(\eta_{1},\Omega_{1}),\ldots, (\eta_{p+1},\Omega_{p+1})$$ are cosymplectic, so is any element of the almost cosymplectic $p$-sphere $\left\{(\eta_{\lambda},\Omega_{\lambda})\right\}_{\lambda\in\mathbb{S}^p}$. In this case we shall speak of
\emph{cosymplectic $p$-sphere}.

Now we describe a class of examples of cosymplectic spheres in dimension $3$.

\begin{example}\label{3manifold}
Let $M$ be any orientable three-dimensional manifold and  let $\eta_{1}$, $\eta_{2}$, $\eta_{3}$ be three $1$-forms which are linearly independent at every point of $M$. Notice that such $1$-forms always exist, by Stiefel's theorem (\cite{stiefel}).
Set
\begin{equation*}
\Omega_{1}:=\eta_{2}\wedge\eta_{3}, \ \  \Omega_{2}:=\eta_{3}\wedge\eta_{1}, \ \ \Omega_{3}:=\eta_{1}\wedge\eta_{2}.
\end{equation*}
Then it is easy to see that $(\eta_{1},\Omega_{1})$, $(\eta_{2},\Omega_{2})$, $(\eta_{3},\Omega_{3})$ are almost cosymplectic structures. Moreover, a direct computation shows that
\begin{equation*}
\left(\lambda_{1}\eta_{1}+\lambda_{2}\eta_{2}+\lambda_{3}\eta_{3}\right)\wedge\left(\lambda_{1}\Omega_{1}+\lambda_{2}\Omega_{2}+\lambda_{3}\Omega_{3}\right)
=\eta_{1}\wedge\eta_{2}\wedge\eta_{3}\neq 0.
\end{equation*}
Therefore we can conclude that $(\eta_{1},\Omega_{1})$, $(\eta_{2},\Omega_{2})$, $(\eta_{3},\Omega_{3})$ generate a taut almost cosymplectic sphere on $M$. Furthermore, such an almost cosymplectic
sphere is cosymplectic if the forms $\eta_{1}$, $\eta_{2}$, $\eta_{3}$ are closed.
In particular, the above construction applies to the $3$-torus $\mathbb{T}^3$ endowed with the  $1$-forms $\eta_{\alpha}:=d\theta_\alpha$, $\alpha\in\left\{1,2,3\right\}$, where $\theta_1, \theta_2, \theta_3$ are the coordinates on $\mathbb{T}^3$.
\end{example}

Now we consider another special case of Example \ref{3manifold}, where the $3$-manifold is the Heisenberg group.

\begin{example}\label{heisenberg}
Let $H$ be the Heisenberg group, whose Lie algebra structure is given by
\begin{equation*}
[e_{1},e_{3}] = 0, \ \ \ [e_{1},e_{2}]=\gamma e_{3}, \ \ \ [e_{3},e_{2}]=0.
\end{equation*}
Let $\eta_{1}$, $\eta_{2}$, $\eta_3$ be the $1$-forms dual to $e_{1}$, $e_{2}$, $e_{3}$, respectively. We have
\begin{equation*}
\d\eta_{1}=\d\eta_{2}=0,\qquad\d\eta_{3}=-\frac{\gamma}{2}\eta_{1}\wedge \eta_{2}.
\end{equation*}

Notice that the resulting sphere is neither contact not cosymplectic, but $(\eta_{1},\Omega_{1})$ and $(\eta_{2},\Omega_{2})$ generate a left-invariant  taut cosymplectic circle on $H$.
\end{example}

The result of Zessin \cite{zessin05} about the non-existence of contact circles in dimension $4n+1$ can be easily generalized to the almost cosymplectic case.

\begin{theorem}\label{4n+1}
Manifolds of dimension $4n+1$ do not admit any almost cosymplectic circle and thus any almost cosymplectic $p$-sphere for $p\geq 2$.
\end{theorem}

\begin{proof}
We will prove the theorem for $n=1$. Higher dimensions can be dealt with similarly.
Assume that $\left\{(\eta_{1},\Omega_{1}),(\eta_{2},\Omega_{2})\right\}$ is an almost cosymplectic circle on a smooth manifold $M$ of dimension $5$. Then, for any
$\lambda=(\lambda_{1},\lambda_{2})\in\mathbb{S}^1$ we have
\begin{equation}\label{5dim}
\begin{aligned}
\eta_{\lambda}\wedge\Omega_{\lambda}^2&=\lambda_{1}^{3}\eta_{1}\wedge\Omega_{1}^{2}+\lambda_{1}^2\lambda_{2}\left(2\eta_{1}\wedge\Omega_{1}\wedge\Omega_{2}+\eta_{2}\wedge\Omega_{1}^{2}\right)\\
&\quad+\lambda_{1}\lambda_{2}^{2}\left(2\eta_{2}\wedge\Omega_{1}\wedge\Omega_{2}+\eta_{1}\wedge\Omega_{2}^2\right)+\lambda_{2}^{3}\eta_{2}\wedge\Omega_{2}^{2}.
\end{aligned}
\end{equation}
Let $p$ be a point of $M$ and let $\left\{e_{1},e_{2},e_{3},e_{4},e_{5}\right\}$ be a basis of $T_{p}M$. Then in view of \eqref{5dim} we have that the real valued function
\begin{equation*}
(\lambda_{1},\lambda_{2})\in\mathbb{R}^2 \mapsto f(\lambda_{1},\lambda_{2}):= \eta_{\lambda}\wedge\Omega_{\lambda}^2 \left(e_{1},e_{2},e_{3},e_{4},e_{5}\right)
\end{equation*}
is the polynomial function of a homogeneous polynomial of degree $3$ in the indeterminates $\lambda_{1}, \lambda_{2}$. Thus, in particular, $f(-\lambda_{1},-\lambda_{2})=-f(\lambda_{1},\lambda_{2})$
for any $(\lambda_{1},\lambda_{2})\in\mathbb{S}^1$. Therefore $f$ should have some zero in $\mathbb{S}^1$, but this contradicts the fact that $(\eta_{\lambda},\Omega_{\lambda})$ is an
almost cosymplectic structure.
\end{proof}

Note that in contrast to the case of dimension $4n+1$, considered in Theorem~\ref{4n+1},
we shall see in Section \ref{3structures} that there are examples of cosymplectic spheres in any dimension $4n+3$.

\begin{example}
Let $(\eta_{1},\Omega_{1})$ and $(\eta_{2},\Omega_{2})$ be the cosymplectic structures on $\mathbb{R}^7$ given by
\begin{gather*}
\eta_{1}:=\d x_7, \ \ \ \Omega_{1}:=\d x_{1}\wedge \d x_{2} +\d x_{3}\wedge \d x_{4} + \d x_{5}\wedge \d x_{6} \\
\eta_{2}:=\d x_6, \ \ \ \Omega_{2}:=(\d x_{1} + \d x_{2})\wedge \d x_{3} + (\d x_{4} + \d x_{5})\wedge \d x_{7} - \d x_{2} \wedge \d x_{5}.
\end{gather*}
A straightforward computation shows that for any $(\lambda_{1},\lambda_{2})\in\mathbb{S}^1$ one has
\begin{equation*}
\left(\lambda_{1}\eta_{1}+\lambda_{2}\eta_{2}\right)\wedge\left(\lambda_{1}\Omega_{1}+\lambda_{2}\Omega_{2}\right)^3=6\left(\left(\lambda_{1}^2-\lambda_{2}^2\right)^2+\lambda_{1}^{2}\lambda_{2}^{2}\right)
\d x_{1}\wedge\dots\wedge\d x_{7}.
\end{equation*}
Thus $(\eta_{1},\Omega_{1})$ and $(\eta_{2},\Omega_{2})$  define a non-taut cosymplectic circle. Moreover, the cosymplectic circle is also not round. In fact, we have $\xi_1=\frac{\partial}{\partial x_7}$, $\xi_2=\frac{\partial}{\partial x_6}$ and for any $\lambda\in \mathbb{S}^1$
\begin{equation*}
i_{\lambda_1\xi_1+\lambda_2\xi_2}(\lambda_1\Omega_1+\lambda_2\Omega_2)=\lambda_1\lambda_2 (i_{\xi_1}\Omega_2+i_{\xi_2}\Omega_1)
=-\lambda_1\lambda_2(d x_4+2d x_5))\neq   0,
\end{equation*}
unless $\lambda_1=0$ or $\lambda_2=0$.
\end{example}

\begin{proposition}\label{novanish}
Let  $\left\{(\eta_{1},\Omega_{1}),(\eta_{2},\Omega_{2})\right\}$ be an almost cosymplectic circle. Then  $i_{\xi_2}\Omega_1$ and  $i_{\xi_1}\Omega_2$ nowhere vanish.
\end{proposition}

\begin{proof}
Suppose that
\begin{equation}\label{reeb1}
(i_{\xi_1}\Omega_2)(p)=0,
\end{equation}
for some $p\in M$.
Now, let $(\eta_\lambda,\Omega_\lambda)$ be the almost cosymplectic structure given by $\eta_\lambda=\lambda_{1}\eta_{1}+\lambda_{2}\eta_{2}$ and
$\Omega_{\lambda}=\lambda_{1}\Omega_{1}+\lambda_{2}\Omega_{2}$, for some $\lambda=(\lambda_{1},\lambda_{2})\in \mathbb{S}^{1}$. Because of \eqref{reeb1} and
$i_{\xi_1}\Omega_1=0$  we have $(i_{\xi_1}\Omega_\lambda)(p)=0$. It follows that $\eta_\lambda(\xi_1)\neq 0$ at $p$. For, if $\eta_\lambda(\xi_1)(p)=0$ one would have
 that $\left(\eta_\lambda \wedge \Omega_\lambda\right) (\xi_1,-)=0$ at $p$, so $(\eta_\lambda,\Omega_\lambda)$ would not be an almost cosymplectic structure. Thus
 $\alpha:=\eta_\lambda(\xi_1)(p)\neq 0$. We put $\xi'_{1}:=\xi_{1}/\alpha$. Then $i_{\xi'_{1}}\Omega_\lambda=0$ and $i_{\xi'_1}\eta_\lambda=1$ at $p$.
 Consequently $\xi'_1(p)=\xi_\lambda(p)$ and we conclude that $\xi_\lambda$ is parallel to $\xi_1$ at $p$, for any  $\lambda\in\mathbb{S}^1$.
 For the almost cosymplectic structure $(-\eta_1,-\Omega_1)$, which belongs to the almost cosymplectic circle, the Reeb vector field is
 $-\xi_1$. Since the function $f\colon\mathbb{S}^1\longrightarrow \mathbb{R}$, defined by $\xi_\lambda(p) = f(\lambda) \xi_1(p)$,
 is continuous and takes the value $-1$ at $(-1,0)$ and the value $1$ at $(1,0)$, there exists some $\lambda_{0}\in\mathbb{S}^1$ such that $f(\lambda_{0})=0$, that is $\xi_{\lambda_{0}}(p)=0$. So we get a
 contradiction. Thus, the $1$-form $i_{\xi_1}\Omega_2$ nowhere vanishes. Similarly one shows that $i_{\xi_2}\Omega_1$ nowhere vanishes.
\end{proof}

\begin{corollary}\label{independence-reeb}
Let  $\left\{(\eta_{1},\Omega_{1}),(\eta_{2},\Omega_{2})\right\}$ be an almost cosymplectic circle. Then the Reeb vector fields $\xi_1$ and $\xi_2$ of $(\eta_{1},\Omega_{1})$ and $(\eta_{2},\Omega_{2})$  are everywhere linearly independent.
\end{corollary}

\begin{definition}\label{reebdistribution}
Given an almost cosymplectic $p$-sphere $ \left\{(\eta_{\lambda},\Omega_{\lambda})\right\}_{\lambda\in\mathbb{S}^p}$, its \emph{Reeb distribution}  $\mathcal{V}$ is the distribution  generated by the Reeb vector fields of the generators
$(\eta_{1},\Omega_{1}), \ldots, (\eta_{p+1},\Omega_{p+1})$, i.e.
\[
\mathcal{V}=  \langle \xi_{1},\ldots,\xi_{p+1} \rangle.
\]
\end{definition}

A natural question concerns the integrability of the Reeb distribution $\mathcal{V}$.
The following proposition shows that in the $3$-dimensional case $\mathcal{V}$ is the common kernel of the $1$-forms $i_{\xi_1}\Omega_2$ and $i_{\xi_2}\Omega_1$.
\begin{proposition}\label{common-kernel}
Let  $\left\{(\eta_{1},\Omega_{1}),(\eta_{2},\Omega_{2})\right\}$ be an almost cosymplectic circle on a 3-dimensional manifold $M^3$. Then the Reeb distribution $\mathcal{V}=  \langle \xi_1, \xi_2\rangle$ is given by
\[
\mathcal{V}=\ker (i_{\xi_1}\Omega_2)=\ker (i_{\xi_2}\Omega_1).
\]
\end{proposition}
\begin{proof}
By Proposition \ref{novanish} the 1-forms $i_{\xi_1}\Omega_2$ and $i_{\xi_2}\Omega_1$ nowhere vanish, thus their kernels have dimension two. Moreover, one can easily check that $\xi_1, \xi_2 \in\ker (i_{\xi_1}\Omega_2)$ and $\xi_1, \xi_2 \in\ker (i_{\xi_2}\Omega_1)$. Since the two Reeb vector fields $\xi_1$, $\xi_{2}$ are linearly independent by Corollary~\ref{independence-reeb}, then $\mathcal{V}$ coincides with $\ker (i_{\xi_1}\Omega_2)$ and $\ker (i_{\xi_2}\Omega_1)$ for dimensional reasons.
\end{proof}

Recall that the kernel of a $1$-form $\theta$ is integrable if and only if $\theta\wedge\d\theta=0$.  Therefore we obtain  the following corollary.
\begin{corollary}\label{integrable-reeb}
Let  $\left\{(\eta_{1},\Omega_{1}),(\eta_{2},\Omega_{2})\right\}$ be an almost cosymplectic circle on a 3-dimensional manifold $M^3$. Then the Reeb vector fields $\xi_1$ and $\xi_2$  generate a two dimensional integrable distribution if and only if
\begin{equation*}
i_{\xi_1}\Omega_2 \wedge \d i_{\xi_1}\Omega_2=0,
\end{equation*}
or equivalently if and only if
\begin{equation*}
i_{\xi_2}\Omega_1 \wedge \d i_{\xi_2}\Omega_1=0.
\end{equation*}
\end{corollary}

\begin{remark}
An example of a taut cosymplectic circle on a 3-dimensional manifold with non-integrable Reeb distribution is given by Example \ref{heisenberg} in the case $\gamma\neq 0$ where the condition $i_{\xi_1}\Omega_2 \wedge \d i_{\xi_1} \Omega_2=0$ of Corollary \ref{integrable-reeb} is not satisfied. In that example we have
$\xi_1=e_1$, $\xi_2=e_2$ and
\[
i_{\xi_1}\Omega_2= i_{\xi_1}(\eta_{3} \wedge \eta_{1})= -  \eta_{3}.
\]
Hence
\[
i_{\xi_1}\Omega_2 \wedge \d i_{\xi_1} \Omega_2=\eta_{3}\wedge \d \eta_{3}= -\frac{\gamma}{2} \eta_{1}\wedge \eta_{2}\wedge \eta_{3}\neq 0.
\]
The distribution generated by $\xi_1$ and $\xi_2$ is indeed not integrable since
\(
[\xi_1,\xi_2]=e_3.
\)
\end{remark}

Concerning the topological properties of a cosymplectic $p$-sphere on a compact manifold we have the following result.

\begin{proposition}\label{betti}
Let $\left\{(\eta_{1},\Omega_{1}),\ldots,(\eta_{p+1},\Omega_{p+1})\right\}$ be a cosymplectic $p$-sphere on a compact manifold $M$. Then the classes
\begin{enumerate}
\item[(i)] $[\eta_{1}],\ldots, [\eta_{p+1}]\in H^1(M)$ and
\item[(ii)] $[\Omega_{1}],\ldots, [\Omega_{p+1}]\in H^2(M)$
\end{enumerate}
are nonzero and linearly independent.
Therefore the following conditions on the Betti numbers are fulfilled: $b_1\geq p+1$, $b_2\geq p+1$.
\end{proposition}
\begin{proof}
Suppose that
\[
\sum_{i=1}^{p+1}\lambda_i [\eta_{i}]=0 \mbox{ in $H^1(M)$.}
\]
Upon dividing by $\sum_{i=1}^{p+1}\lambda_i^2>0$, we can assume that $(\lambda_{1},\ldots,\lambda_{p+1})\in\bS^{p}$.
Then, there is a function $f$ such that
\[
\d f= \sum_{i=1}^{p+1}\lambda_i \eta_{i}.
\]
On the other hand, we know that $(\sum_{i=1}^{p+1}\lambda_i \eta_{i}, \sum_{i=1}^{p+1}\lambda_i \Omega_{i})$ is a cosymplectic structure on $M$.
Thus,
\begin{align*}
\d \left(f \wedge (\sum_{i=1}^{p+1}\lambda_i \Omega_{i})^n\right)=  \d f \wedge \left(\sum_{i=1}^{p+1} \lambda_i \Omega_{i}\right)^n=
\left(\sum_{i=1}^{p+1}\lambda_i \eta_{i}\right)\wedge \left(\sum_{i=1}^{p+1} \lambda_i \Omega_{i}\right)^n
\end{align*}
should have a nontrivial  class in $H^{2n+1}(M)$. We get a contradiction, hence \allowbreak $[\eta_{1}],\ldots, [\eta_{p+1}]\in H^1(M)$ are nonzero and linearly independent. Similarly one shows that $[\Omega_{1}],\ldots, [\Omega_{p+1}]\in H^2(M)$ are nonzero and linearly independent.
\end{proof}

\bigskip

Now we give a full classification of $3$-dimensional compact manifolds admitting a cosymplectic circle.

\begin{theorem}
Let $M$ be a compact $3$-dimensional manifold. $M$ admits a cosymplectic circle if and only if $M$ is either a $3$-torus or a quotient of the Heisenberg group by a co-compact subgroup.
\end{theorem}
\begin{proof}
Assume that $M$ carries a cosymplectic circle ${(\eta_1,\Omega_1), (\eta_2,\Omega_2)}$. Since $\eta_1, \eta_2$ are closed and linearly independent, by Tischler's theorem
\cite[Corollary 2]{tischler}   $M$ is a locally trivial fibration  over the $2$-torus
$$F\longrightarrow M \longrightarrow \mathbb{T}^2.$$
The $1$-dimensional fiber $F$ is closed, hence compact. Thus it is a finite disjoint union of circles. By eventually passing  to a finite cover of the base, we can assume that $F \cong \mathbb{S}^1$. Indeed, one can show that every fibration of a connected space over a connected base can be written as a fibration with a connected fiber followed by a covering (cf. Stein factorization in algebraic geometry). Thus we have
$$M\longrightarrow B \longrightarrow \mathbb{T}^2,$$
where $M\to B$ is an $\bS^1$-fiber bundle and $B$ is a finite covering of $\mathbb{T}^2$, and hence  $B\cong \mathbb{T}^2$. Therefore $M$ is an $\bS^1$-bundle over $\mathbb{T}^2$.
Moreover, this $\bS^1$-bundle over $\mathbb{T}^2$ is oriented since $M$ is oriented.
Now, every oriented $\bS^1$-bundle is principal by \cite[Proposition 6.15]{morita}. Hence $M$ is a principal circle bundle over $\mathbb{T}^2$ and by \cite[Theorem 3]{palais} we get that $M$ is diffeomorphic to a $2$-step nilpotent nilmanifold $G/\Gamma$. However, every $3$-dimensional $2$-step nilpotent Lie group is either the Heisenberg group or the torus~$\mathbb{T}^3$.

Conversely, Example \ref{heisenberg} and Example \ref{3manifold} show that both $H/ \Gamma$ and $\mathbb{T}^3$ admit a cosymplectic circle.
\end{proof}

\section{Taut and round cosymplectic $p$-spheres}

In this section we focus on round and taut cosymplectic $p$-spheres. We start with a characterization of roundness.

\begin{proposition}\label{round1}
Let $\left\{(\eta_{1},\Omega_{1}),\ldots,(\eta_{p+1},\Omega_{p+1})\right\}$ be a cosymplectic $p$-sphere on a manifold $M$ and let $\xi_1, \ldots,\xi_{p+1}$ denote
the  Reeb vector fields of the generators. Then $\left\{(\eta_{1},\Omega_{1}),\ldots,(\eta_{p+1},\Omega_{p+1})\right\}$ is round if and only if the
following conditions are fulfilled
\begin{enumerate}
\item[(i)] $\eta_{i}(\xi_{j})+\eta_{j}(\xi_{i}) = 0$ \ for any $i,j\in\left\{1,\ldots,p+1\right\}$, $i\neq j$
\item[(ii)] $i_{\xi_{i}}\Omega_{j}+i_{\xi_{j}}\Omega_{i}=0$ \ for any $i,j\in\left\{1,\ldots,p+1\right\}$.
\end{enumerate}
\end{proposition}
\begin{proof}
The cosymplectic sphere $\left\{(\eta_{1},\Omega_{1}),\ldots,(\eta_{p+1},\Omega_{p+1})\right\}$  is round if and only if the vector field
$\xi_{\lambda}:=\lambda_{1}\xi_{1}+\ldots+\lambda_{p+1}\xi_{p+1}$ is the Reeb vector field of $(\eta_{\lambda},\Omega_{\lambda})$, i.e. $\eta_{\lambda}(\xi_\lambda)=1$ and $i_{\xi_\lambda}\Omega_\lambda=0$. These conditions give
\begin{align*}
\sum_{i=1}^{p+1}\lambda_{i}^{2} + \sum_{i\neq j}\lambda_{i}\lambda_{j}\eta_{i}(\xi_j)=1,\qquad \sum_{i\neq j} \lambda_{i}\lambda_{j}i_{\xi_i}\Omega_j=0.
\end{align*}
Substituting $\lambda_{k}=0$ for $k\neq i, j$  and $\lambda_{i}=\lambda_{j}=1/\sqrt{2}$, we get
\begin{equation*}
\frac12\left(\eta_{i}(\xi_j)+\eta_{j}(\xi_i)\right)=0,\qquad  \frac12(i_{\xi_i}\Omega_j+i_{\xi_j}\Omega_i)=0.
\end{equation*}
\end{proof}

We now prove that in dimension 3 tautness and roundness are equivalent. We start with the following lemma, which is a characterization of tautness in
dimension $3$.

\begin{lemma}\label{tautness}
A cosymplectic circle $\left\{(\eta_{1},\Omega_{1}),(\eta_{2},\Omega_{2})\right\}$  on a $3$-manifold $M$ is taut if and only if the
following conditions are fulfilled
\begin{gather}
\eta_{1}\wedge\Omega_{1}=\eta_{2}\wedge\Omega_{2} \label{taut1}\\
\eta_{1}\wedge\Omega_{2}=-\eta_{2}\wedge\Omega_{1} \label{taut2}.
\end{gather}
\end{lemma}
\begin{proof}
Let us assume that $\left\{(\eta_{1},\Omega_{1}),(\eta_{2},\Omega_{2})\right\}$  is taut, that is
\begin{equation*}
\left(\lambda_{1}\eta_{1}+\lambda_{2}\eta_{2}\right)\wedge\left(\lambda_{1}\Omega_{1}+\lambda_{2}\Omega_{2}\right)=\left(\lambda'_{1}\eta_{1}+\lambda'_{2}\eta_{2}\right)\wedge\left(\lambda'_{1}\Omega_{1}+\lambda'_{2}\Omega_{2}\right)
\end{equation*}
for any $(\lambda_{1},\lambda_{2}),(\lambda'_{1},\lambda'_{2})\in\mathbb{S}^1$. Then by taking $(\lambda_{1},\lambda_{2})=(1,0)$ and
$(\lambda'_{1},\lambda'_{2})=(0,1)$ we get \eqref{taut1}. Next, by taking $(\lambda_{1},\lambda_{2})=\left(\frac{1}{\sqrt{2}},\frac{1}{\sqrt{2}}\right)$
and $(\lambda'_{1},\lambda'_{2})=(1,0)$, and using \eqref{taut1}, we obtain \eqref{taut2}. \ Conversely, if \eqref{taut1}--\eqref{taut2} hold, then for any
$(\lambda_{1},\lambda_{2})\in\mathbb{S}^1$ we have
\begin{align*}
\left(\lambda_{1}\eta_{1}+\lambda_{2}\eta_{2}\right)\wedge\left(\lambda_{1}\Omega_{1}+\lambda_{2}\Omega_{2}\right)&=\lambda_{1}^{2}\eta_{1}\wedge\Omega_{1}
+\lambda_{1}\lambda_{2}\eta_{1}\wedge\Omega_{2}+\lambda_{2}\lambda_{1}\eta_{2}\wedge\Omega_{1}\\
&\quad+\lambda_{2}^{2}\eta_{2}\wedge\Omega_{2}\\
&=(\lambda_{1}^{2}+\lambda_{2}^{2})\eta_{1}\wedge\Omega_{1}+\lambda_{1}\lambda_{2}\eta_{1}\wedge\Omega_{2}+\lambda_{2}\lambda_{1}\eta_{2}\wedge\Omega_{1}\\
&=\eta_{1}\wedge\Omega_{1}.
\end{align*}
This proves that $\left\{(\eta_{1},\Omega_{1}),(\eta_{2},\Omega_{2})\right\}$ is taut.
\end{proof}

\begin{theorem}\label{roundtaut}
On a  $3$-manifold $M$ a cosymplectic $p$-sphere is taut if and only if it is round.
\end{theorem}
\begin{proof}
We prove the statement of the theorem only for cosymplectic circles, since the case of cosymplectic spheres can be proved in a very similar way.
Let $\left\{(\eta_{1},\Omega_{1}),\right.$ $\left.(\eta_{2},\Omega_{2})\right\}$ be a taut cosymplectic circle, that is \eqref{taut1}--\eqref{taut2} hold. Let $\xi_1$ and
$\xi_2$ be the Reeb vector fields of $(\eta_{1},\Omega_{1})$ and $(\eta_{2},\Omega_{2})$, respectively.  From \eqref{taut1} and \eqref{reeb} it follows
that
\begin{equation}\label{condition1}
i_{\xi_2}\Omega_1=i_{\xi_2}i_{\xi_1}(\eta_{1}\wedge\Omega_1)=i_{\xi_2}i_{\xi_1}(\eta_{2}\wedge\Omega_2)=-i_{\xi_1}i_{\xi_2}(\eta_{2}\wedge\Omega_2)=-i_{\xi_1}\Omega_2.
\end{equation}
Moreover, by applying the interior product by $\xi_1$ to \eqref{taut2} we get
\begin{equation}\label{intermediate1}
\Omega_2 - \eta_{1}\wedge i_{\xi_1}\Omega_2 = -i_{\xi_1}\eta_2 \wedge\Omega_1.
\end{equation}
By applying the interior product by $\xi_2$ to both sides of \eqref{intermediate1} we get
\begin{equation}\label{intermediate2}
- i_{\xi_2}\eta_{1}\wedge i_{\xi_1}\Omega_2 = - i_{\xi_1}\eta_{2}\wedge i_{\xi_{2}}\Omega_{1}.
\end{equation}
Thus, by using \eqref{condition1} in the first side of \eqref{intermediate2} we find
\begin{equation*}
\left(i_{\xi_1}\Omega_{2}\right)\left(\eta_{1}(\xi_2)+\eta_{2}(\xi_1)\right)=0,
\end{equation*}
which, by Proposition \ref{novanish}, is equivalent to
\begin{equation}\label{condition2}
\eta_{1}(\xi_2)+\eta_{2}(\xi_1)=0.
\end{equation}
Notice that \eqref{condition2} and \eqref{condition1} are, respectively, the conditions (i) and (ii) in Proposition \ref{round1}. Therefore the
cosymplectic circle is round.

Conversely, let $\left\{(\eta_{1},\Omega_{1}),(\eta_{2},\Omega_{2})\right\}$  be a round cosymplectic circle.  Let $\xi_1$
and $\xi_2$ be the Reeb vector fields of the generators $(\eta_{1},\Omega_{1})$ and $(\eta_{2},\Omega_{2})$, respectively. By Proposition \ref{round1} we
have
\begin{gather}
\eta_{1}(\xi_2)=-\eta_{2}(\xi_1),\label{intermediate4}\\
i_{\xi_2}\Omega_1 =- i_{\xi_1}\Omega_2  \label{intermediate3}
\end{gather}
Moreover, by Corollary \ref{independence-reeb} one has that $\xi_1$ and $\xi_2$ are linearly independent. Thus we can complete them to  a local basis $(\xi_1,
\xi_2, Z)$ defined  on an open subset $U$ of $M$. Then by \eqref{intermediate3} we get
\begin{equation}\label{intermediate5}
\Omega_1(\xi_2,Z)=-\Omega_2(\xi_1,Z).
\end{equation}
Then, by using \eqref{intermediate4} and  \eqref{intermediate5} we have,  for any $\lambda=(\lambda_1,\lambda_2)\in\mathbb{S}^1$,
\begin{align*}
\eta_{\lambda}\wedge&\Omega_{\lambda}(\xi_1,\xi_2,Z)=(\lambda_1\eta_1+\lambda_2\eta_2)\wedge(\lambda_1\Omega_1+\lambda_2\Omega_2)(\xi_1,\xi_2,Z)\\
&=\lambda_1^2\, \eta_{1}\wedge\Omega_1 (\xi_1,\xi_2,Z) + \lambda_1\lambda_2 \,\eta_{1}\wedge\Omega_2 (\xi_1,\xi_2,Z) + \lambda_2\lambda_1\, \eta_{2}\wedge\Omega_1 (\xi_1,\xi_2,Z)
\\&\quad+\lambda_2^2 \,\eta_{2}\wedge\Omega_2 (\xi_1,\xi_2,Z)\\
&=\lambda_{1}^{2}\,\eta_1({\xi_1})\Omega_{1}(\xi_{2},Z)+\lambda_{1}\lambda_{2}\,\eta_{1}(\xi_2)\Omega_{2}(Z,\xi_1)+\lambda_{2}\lambda_{1}\,\eta_{2}(\xi_1)\Omega_{1}(\xi_2,Z)
\\&\quad+\lambda_{2}^2\,\eta_{2}(\xi_2)\Omega_{2}(Z,\xi_1)\\
&=\left(\lambda_{1}^2+\lambda_{2}^2\right)\Omega_{1}(\xi_{2},Z)
=\eta_{1}\wedge\Omega_{1}(\xi_{1},\xi_{2},Z).
\end{align*}
This proves that the cosymplectic circle is taut.
\end{proof}

Theorem \ref{roundtaut} does not hold in higher dimensions, as it is shown in the following example.

\begin{example}
Let us consider the cosymplectic structures $(\eta_{1},\Omega_{1})$, $(\eta_{2},\Omega_{2})$ on $\mathbb{T}^7$ given by
\begin{align*}
\eta_1:&=d x_7,\qquad \Omega_1:=d x_1 \wedge d x_2 + d x_3 \wedge d x_4 + d x_5 \wedge d x_6,\\
\eta_2:&=d x_2,\qquad \Omega_2:=d x_5 \wedge d x_4 - d x_3 \wedge d x_6 + (d x_1 + d x_3)\wedge d x_7.
\end{align*}
One can prove that they generate a cosymplectic circle on $\mathbb{T}^7$ which is taut but not round, since
the condition (ii) in Proposition \ref{round1} is not satisfied. On the other hand, the cosymplectic circle generated by
\begin{align*}
\eta_1:&=d x_7,\quad \eta_2:=-d x_2, \qquad\Omega_1:=d x_1 \wedge d x_2 + d x_3 \wedge d x_4 + d x_5 \wedge d x_6,\\
\Omega_2:&=d x_3 \wedge (d x_5 + d x_6) + d x_4 \wedge d x_5 + (d x_1 + d x_3)\wedge d x_6 + d x_1 \wedge d x_7,
\end{align*}
is round but not taut.
\end{example}

To any taut cosymplectic circle $\left\{(\eta_{1},\Omega_{1}),(\eta_{2},\Omega_{2})\right\}$ on a $3$-dimensional smooth manifold $M$ it can be associated a complex structure $J$ on $M\times\mathbb{R}$ in a canonical way.
Indeed given an almost cosymplectic structure  $(\eta,\Omega)$ on $M$ we can define an almost symplectic structure (i. e. a nondegenerate $2$-form) on $M\times\mathbb{R}$ by
\begin{equation*}
\omega:=dt\wedge \eta+\Omega.
\end{equation*}
It is well known that $(M,\eta,\Omega)$ is cosymplectic if and only if $(M\times\mathbb{R},\omega)$ is symplectic.
Recall that a pair of  symplectic structures $(\omega_1,\omega_2)$ on an oriented 4-manifold is  said to be a \emph{symplectic couple} \cite{geiges96} if one has $\omega_1 \wedge \omega_2\equiv 0$ and $\omega_1^2$, $\omega_2^2$ are volume forms defining the positive orientation. A symplectic couple is called \emph{conformal} if $\omega_1^2=\omega_2^2$.  It is convenient to extend the definition to the case of nondegenerate 2-forms. In that case one simply speaks of a \emph{couple} or a \emph{conformal couple}, respectively, omitting the term `symplectic'. Two couples are called \emph{equivalent} if at every point they span the same oriented plane of nondegenerate 2-forms.
Now, the cosymplectic structures $(\eta_{1},\Omega_{1})$ and $(\eta_{2},\Omega_{2})$ give rise to two symplectic structures $\omega_1$ and $\omega_2$ on $M\times\mathbb{R}$. They satisfy the following relations.
\begin{align}\label{couple}
\omega_1 \wedge \omega_1 &= 2 dt \wedge\eta_1 \wedge \Omega_1\\
\omega_2 \wedge \omega_2 &= 2 dt \wedge\eta_2 \wedge \Omega_2\\
\omega_1 \wedge \omega_2 &= 2 dt \wedge(\eta_1 \wedge \Omega_2 + \eta_2 \wedge \Omega_1).
\end{align}
As a consequence, we have $\omega_1 \wedge \omega_1=\omega_2 \wedge \omega_2$ if and only if $\eta_1 \wedge \Omega_1=\eta_2 \wedge \Omega_2$.

Hence by Lemma \ref{tautness} we have that $\left(\omega_1,\omega_2\right)$ is a conformal symplectic couple on $M\times\mathbb{R}$  if and only if $\left\{(\eta_{1},\Omega_{1}),(\eta_{2},\Omega_{2})\right\}$ is a taut cosymplectic circle on $M$.
Therefore, by \cite[Theorem 2.2]{geiges96} to any taut cosymplectic circle $\left\{(\eta_{1},\Omega_{1}),(\eta_{2},\Omega_{2})\right\}$ on a $3$-dimensional smooth manifold $M$, we can associate a unique complex structure $J$ on $M\times\mathbb{R}$ defined by the property that the forms of type $(2,0)$ with respect to $J$ are precisely those of the type  $\tilde\omega_1+ i \tilde\omega_2$,  where $(\tilde\omega_1,\tilde\omega_2)$ is any conformal couple equivalent to $(\tilde\omega_1,\tilde\omega_2)$.
The obtained complex structure $J$ is a recursion operator in the sense of \cite{bande}, i.e. it is the
unique endomorphism such that  $i_{X}\omega_1=i_{J X}\omega_2$ for any $X\in TM$.

\section{Cosymplectic spheres and $3$-structures}\label{3structures}
In this section we describe a wide class of examples of almost cosymplectic spheres on a manifold of dimension $4n+3$.

Recall that given an almost cosymplectic structure $(\eta,\Omega)$ there exist a Riemannian metric $g$ and a tensor field $\phi$ such that
\begin{gather}
\phi^2=-I+\eta\otimes\xi, \label{almostcontactmetric0}\\
 \Omega=g(-,\phi -).\label{almostcontactmetric1}
\end{gather}
Then one can prove that $\eta\circ\phi=0$,  $\phi\xi=0$ and
\begin{equation}\label{almostcontactmetric2}
g(\phi X, \phi Y) = g(X,Y) - \eta(X)\eta(Y)
\end{equation}
for any $X,Y\in\Gamma(TM)$. An alternative approach consists in taking \eqref{almostcontactmetric0} as definition.
So one defines an \emph{almost contact structure} (unfortunately, the name is rather misleading, but it is  widely used in literature) as the triplet $(\phi, \xi, \eta)$ satisfying \eqref{almostcontactmetric0} and  $\eta(\xi)=1$. In that case the dimension of $M$ is necessarily odd. Then, a Riemannian metric $g$ is called \emph{compatible} if \eqref{almostcontactmetric2} is satisfied.  The geometric structure $(\phi,\xi,\eta,g)$ is called \emph{almost contact metric structure}.  It follows that $g(X,\phi Y)=-g(\phi X, Y)$, so that the bilinear form $\Omega(X,Y)=g(X,\phi Y)$ is a $2$-form, usually called the \emph{fundamental $2$-form}. Then one can prove that $\eta\wedge\Omega^{n}\neq 0$, where $\dim(M)=2n+1$. For further details we refer the reader to \cite{blairbook} and \cite{survey}.

Now, when on the same manifold $M$ there are given three distinct almost
contact structures  $\left(\phi_1,\xi_1,\eta_1\right)$,
$\left(\phi_2,\xi_2,\eta_2\right)$,
$\left(\phi_3,\xi_3,\eta_3\right)$ satisfying the following
relations, for any even permutation
$\left(\alpha,\beta,\gamma\right)$ of $\left\{1,2,3\right\}$,
\begin{equation} \label{3-structure}
\begin{split}
\phi_\gamma=\phi_{\alpha}\phi_{\beta}-\eta_{\beta}\otimes\xi_{\alpha}=-\phi_{\beta}\phi_{\alpha}+\eta_{\alpha}\otimes\xi_{\beta},\quad\\
\xi_{\gamma}=\phi_{\alpha}\xi_{\beta}=-\phi_{\beta}\xi_{\alpha}, \ \ \eta_{\gamma}= \eta_{\alpha}\circ\phi_{\beta}=-\eta_{\beta}\circ\phi_{\alpha},
\end{split}
\end{equation}
we say that $(\phi_\alpha,\xi_\alpha,\eta_\alpha)_{\alpha\in\left\{1,2,3\right\}}$ is an \emph{almost contact $3$-structure} on $M$.
One proves that the conditions \eqref{3-structure} force the dimension of $M$ to be necessarily  $4n+3$ for some integer $n$.
This notion  was introduced independently by Kuo (\cite{kuo}) and Udriste
(\cite{udriste}). Kuo proved also that given an almost contact $3$-structure
$(\phi_\alpha,\xi_\alpha,\eta_\alpha)_{\alpha\in\left\{1,2,3\right\}}$, there exists a
Riemannian metric $g$ compatible with each  almost
contact structure and hence we can speak of \emph{almost contact
metric $3$-structure}. It is well known that in any almost
$3$-contact metric manifold the Reeb vector fields
$\xi_1,\xi_2,\xi_3$ are orthonormal with respect to the compatible
metric $g$. Moreover, by putting
${\mathcal{H}}=\bigcap_{\alpha=1}^{3}\ker\left(\eta_\alpha\right)$
we obtain a codimension $3$ distribution on $M$ and the tangent
bundle splits as the orthogonal sum
$TM={\mathcal{H}}\oplus{\mathcal{V}}$, where ${\mathcal
V}=\left\langle\xi_1,\xi_2,\xi_3\right\rangle$. The distributions
$\mathcal H$ and $\mathcal V$ are called, respectively,
\emph{horizontal} and \emph{Reeb distribution}.

An almost $3$-contact manifold $M$ is said to be
\emph{hyper-normal} if each almost contact structure
$\left(\phi_\alpha,\xi_\alpha,\eta_\alpha\right)$ is normal. An important class of hyper-normal almost contact $3$-structures is given by the $3$-quasi-Sasakian ones.
A \emph{$3$-quasi-Sasakian structure} is  an almost contact metric
$3$-structure such that each structure
$(\phi_\alpha,\xi_\alpha,\eta_\alpha,g)$ is quasi-Sasakian, i.e. it is normal and the corresponding fundamental $2$-form is closed.
Remarkable subclasses of $3$-quasi-Sasakian manifolds are  $3$-Sasakian and $3$-cosymplectic
manifolds.

Many results on
$3$-quasi-Sasakian manifolds were obtained in \cite{agag08} and
\cite{ijm09}. We collect some of them in the following theorem.

\begin{theorem}[\cite{agag08, ijm09}]\label{principale}
Let $(M,\phi_\alpha,\xi_\alpha,\eta_\alpha,g)$ be a $3$-quasi-Sasakian manifold of
dimension $4n+3$. Then, for any even permutation
$(\alpha,\beta,\gamma)$ of $\left\{1,2,3\right\}$, the Reeb vector
fields satisfy
\begin{equation}\label{lie}
 [\xi_\alpha,\xi_\beta]=c\xi_\gamma,
\end{equation}
for some $c\in\mathbb{R}$ which is zero if and only if the manifold is $3$-cosymplectic. Moreover, the $1$-forms $\eta_1$,
$\eta_2$, $\eta_3$ have the same rank,  called the \emph{rank} of
the $3$-quasi-Sasakian manifold $M$. The rank of $M$ is $1$ if and
only if $M$ is $3$-cosymplectic and it is an integer of the form
$4l+3$, for some $l\leq n$, in the other cases.
\end{theorem}

Recall that in a quasi-Sasakian manifold of dimension $2n+1$ and of rank $2p+1$, the
characteristic distribution
\begin{equation}\label{eq:C}
{\mathcal C}:=\left\{X\in TM \,|\, i_{X}\eta=0, i_{X}d\eta=0 \right\}
\end{equation}
is integrable and has dimension $2(n-p)$. Moreover, in a $3$-quasi-Sasakian manifold of dimension $4n+3$ and of rank $4l+3$ we can also consider the
distribution
\begin{equation}\label{eq:E}
{\mathcal E}:=\left\{X\in TM \,|\, i_{X}\eta_{\alpha}=0, i_{X}d\eta_{\alpha}=0 \ \hbox{for any $\alpha=1,2,3$} \right\}
\end{equation}
which turns out to be integrable and to have dimension $4(n-l)$.

\medskip

Recall that given two tensor fields $P$ and $Q$ of type $(1,1)$ on a smooth manifold $M$, one can
define a tensor field of type $(1,2)$, denoted by $[P,Q]$, by setting
\begin{align*}
[P,Q](X,Y)&=[PX,QY]-P[QX,Y]-Q[X,PY]+[QX,PY]-Q[PX,Y]\\
&\quad -P[X,QY]+(PQ+QP)[X,Y],
\end{align*}
where $X$ and $Y$ are arbitrary vector fields on $M$.  The tensor
$[P,Q]$ is usually called the \emph{Nijenhuis concomitant}  of $P$
and $Q$.

Now let  $(\phi_\alpha,\xi_\alpha,\eta_\alpha)_{\alpha\in\left\{1,2,3\right\}}$, be an almost contact $3$-structure
on $M$. For any $\alpha,\beta\in\left\{1,2,3\right\}$ we define four
tensors in the following way
\begin{gather}
N^{(1)}_{\alpha,\beta}:=[\phi_{\alpha},\phi_{\beta}]+d\eta_{\alpha}\otimes\xi_{\beta} + d\eta_{\beta}\otimes\xi_{\alpha} \label{uno} \\
N^{(2)}_{\alpha,\beta}(X,Y):=\left({\mathcal L}_{\phi_\alpha
X}\eta_{\beta}\right)(Y)-\left({\mathcal L}_{\phi_\alpha
Y}\eta_{\beta}\right)(X) +
\left({\mathcal L}_{\phi_\beta X}\eta_{\alpha}\right)(Y)-\left({\mathcal L}_{\phi_\beta Y}\eta_{\alpha}\right)(X) \label{due} \\
N^{(3)}_{\alpha,\beta}:={\mathcal L}_{\xi_\alpha}\phi_\beta + {\mathcal L}_{\xi_\beta}\phi_\alpha \label{tre} \\
N^{(4)}_{\alpha,\beta}:={\mathcal L}_{\xi_\alpha}\eta_\beta +
{\mathcal L}_{\xi_\beta}\eta_\alpha. \label{quattro}
\end{gather}
These tensors satisfy
$N^{(i)}_{\alpha,\beta}=N^{(i)}_{\beta,\alpha}$, $1\leq i \leq 4$. Moreover, $N^{(1)}_{\alpha,\alpha}=N^{(1)}_{\phi_\alpha}$ and
$2 N^{(i)}_{\alpha,\alpha}=N^{(i)}_{\phi_\alpha}$, $2\leq i \leq 4$, where
$N^{(i)}_{\phi_\alpha}$ are the  fundamental tensors of an almost
contact manifold (cf. \cite[(1.3)]{blair0}).

The following theorem establishes a general property of hyper-normal almost contact
$3$-structures.

\begin{theorem}\label{concomitant}
Let $(\phi_\alpha,\xi_\alpha,\eta_\alpha)_{\alpha\in\left\{1,2,3\right\}}$ be a hyper-normal almost contact
$3$-structure on $M$. Then for each
$\alpha,\beta\in\left\{1,2,3\right\}$ the tensors
$N^{(1)}_{\alpha,\beta}$, $N^{(2)}_{\alpha,\beta}$,
$N^{(3)}_{\alpha,\beta}$, $N^{(4)}_{\alpha,\beta}$ vanish.
\end{theorem}
\begin{proof}
For each $\alpha\in\left\{1,2,3\right\}$ we define  on the product
$M\times \mathbb{R}$ the $(1,1)$-tensor
\begin{equation*}
J_\alpha\left(X, f \frac{d}{dt}\right)=\left(\phi_\alpha X - f\xi_\alpha, \eta_\alpha(X)\frac{d}{dt}\right)
\end{equation*}
where $X$ is a vector field on $M$, $f$ a smooth function on
$M\times\mathbb{R}$ and $t$ denotes the coordinate function on
$\mathbb{R}$. It is well-known that $(J_1,J_2,J_3)$ defines a
hyper-complex structure on $M\times \mathbb{R}$. Now let us fix
$\alpha,\beta \in \left\{1,2,3\right\}$ and let us evaluate the
concomitant $[J_\alpha,J_\beta]$ of the vector fields of type
$((X,0),(Y,0))$ and $((X,0),(0,\frac{d}{dt}))$, where $X,Y$ are
vector fields on $M$. After a  long computation,
by using \eqref{3-structure} and \eqref{uno}--\eqref{quattro}, we get
\begin{gather}
[J_\alpha, J_\beta]\left((X,0),(Y,0)\right)=\left(N^{(1)}_{\alpha,\beta}(X,Y),N^{(2)}_{\alpha,\beta}(X,Y)\frac{d}{dt}\right), \label{formula1} \\
[J_\alpha,
J_\beta]\left(\left(X,0\right),\left(0,\frac{d}{dt}\right)\right)=\left((N^{(3)}_{\alpha,\beta}(X),N^{(4)}_{\alpha,\beta}(X)\frac{d}{dt}\right)
\label{formula2}.
\end{gather}
Because of \cite[Theorem 3.1]{yano-ako73} each Nijenhuis concomitant
$[J_\alpha,J_\beta]$ vanishes and so from
\eqref{formula1}--\eqref{formula2} the vanishing of
$N^{(1)}_{\alpha,\beta}$, $N^{(2)}_{\alpha,\beta}$,
$N^{(3)}_{\alpha,\beta}$, $N^{(4)}_{\alpha,\beta}$ follows.
\end{proof}


\begin{theorem}\label{normal}
Let $(\phi_\alpha,\xi_\alpha,\eta_\alpha,g)_{\alpha\in\left\{1,2,3\right\}}$ be an almost contact
metric $3$-structure on $M$. Then, for any $\lambda:=(\lambda_1,\lambda_2,\lambda_3)\in \mathbb{S}^{2}$  the tensors
\begin{equation}\label{acs}
\phi_\lambda:=\lambda_1 \phi_1 + \lambda_2 \phi_2 + \lambda_3 \phi_3, \ \ \,
\xi_\lambda:=\lambda_1 \xi_1 + \lambda_2 \xi_2 + \lambda_3 \xi_3, \ \ \
\eta_\lambda:=\lambda_1 \eta_1 + \lambda_2 \eta_2 + \lambda_3 \eta_3,
\end{equation}
define an almost contact structure on $M$, compatible with the Riemannian metric $g$.  Furthermore,   if
$(\phi_\alpha,\xi_\alpha,\eta_\alpha)_{\alpha\in\left\{1,2,3\right\}}$ is hyper-normal, then
$(\phi_\lambda,\xi_\lambda,\eta_\lambda,g)$ is a normal almost contact metric structure on
$M$.
\end{theorem}
\begin{proof}
First, by using the relations
$\eta_\alpha(\xi_\beta)=\delta_{\alpha\beta}$ we get
\begin{equation*}
\eta_{\lambda}(\xi_{\lambda})=\lambda_{1}^{2}\eta_{1}(\xi_{1})+\lambda_{2}^{2}\eta_{2}(\xi_{2})+\lambda_{3}^{2}\eta_{3}(\xi_{3})=1.
\end{equation*}
Next, due to \eqref{3-structure}, we have that
\begin{align*}
\phi^2_{\lambda}&= \sum_{\alpha=1}^{3}\lambda_{\alpha}^{2}\phi_{\alpha}^{2}+ \sum_{\alpha<\beta} \lambda_{\alpha}\lambda_{\beta}(\phi_{\alpha}\phi_{\beta}+\phi_{\beta}\phi_{\alpha})\\
&= \sum_{\alpha=1}^{3} (-\lambda_{\alpha}^{2} I + \lambda_{\alpha}^{2} \eta_{\alpha}\otimes\xi_{\alpha}) +
\sum_{\alpha<\beta}\lambda_{\alpha}\lambda_{\beta}(\eta_{\alpha}\otimes\xi_{\beta}+ \eta_{\beta}\otimes \xi_{\alpha})
\\ &= - I + \sum_{\alpha=1}^{3} (\lambda_{\alpha}\eta_{\alpha})\otimes (\lambda_{\alpha}\xi_{\alpha})
 + \sum_{\alpha\neq \beta} (\lambda_{\alpha}\eta_{\alpha})\otimes (\lambda_{\beta}\xi_{\beta})
\\ &= - I + \eta_{\lambda} \otimes \xi_{\lambda}.
 \end{align*}


Finally, we prove that the almost contact structure
$(\phi_{\lambda},\xi_{\lambda},\eta_{\lambda})$ is compatible with $g$. Indeed, for any
$X,Y\in\Gamma(TM)$ one has
\begin{align*}
g(\phi_{\lambda} X,\phi_{\lambda} Y)&=\lambda_{1}^{2}g(\phi_{1}X,\phi_{1}Y)+\lambda_{2}^{2}g(\phi_{2}X,\phi_{2}Y)+\lambda_{3}^{2}g(\phi_{3}X,\phi_{3}Y)\\
&\quad+\lambda_{1}\lambda_{2}g(\phi_{1}X,\phi_{2}Y)+\lambda_{1}\lambda_{3}g(\phi_{1}X,\phi_{3}Y)+\lambda_{1}\lambda_{2}g(\phi_{2}X,\phi_{1}Y)\\&\quad+\lambda_{2}\lambda_{3}g(\phi_{2}X,\phi_{3}Y)
+\lambda_{1}\lambda_{3}g(\phi_{3}X,\phi_{1}Y)+\lambda_{2}\lambda_{3}g(\phi_{3}X,\phi_{2}Y)\\
&=(\lambda_{1}^{2}+\lambda_{2}^{2}+\lambda_{3}^{2})g(X,Y)-\lambda_{1}^{2}\eta_{1}(X)\eta_{1}(Y)-\lambda_{2}^{2}\eta_{2}(X)\eta_{2}(Y)\\
&\quad-\lambda_{3}^{2}\eta_{3}(X)\eta_{3}(Y)
-\lambda_{1}\lambda_{2}g(X,\eta_{1}(Y)\xi_{2})-\lambda_{1}\lambda_{2}g(X,\eta_{2}(Y)\xi_{1})\\
&\quad-\lambda_{1}\lambda_{3}g(X,\eta_{1}(Y)\xi_{3})-\lambda_{1}\lambda_{3}g(X,\eta_{3}(Y)\xi_{1})-\lambda_{2}\lambda_{3}g(X,\eta_{2}(Y)\xi_{3})\\
&\quad-\lambda_{2}\lambda_{3}g(X,\eta_{3}(Y)\xi_{2})\\
&=g(X,Y)-\lambda_{1}^{2}\eta_{1}(X)\eta_{1}(Y)-\lambda_{2}^{2}\eta_{2}(X)\eta_{2}(Y)-\lambda_{3}^{2}\eta_{3}(X)\eta_{3}(Y)\\
&\quad-\lambda_{1}\lambda_{2}\eta_{2}(X)\eta_{1}(Y)-\lambda_{1}\lambda_{2}\eta_{1}(X)\eta_{2}(Y)-\lambda_{1}\lambda_{3}\eta_{3}(X)\eta_{1}(Y)\\
&\quad-\lambda_{1}\lambda_{3}\eta_{1}(X)\eta_{3}(Y)-\lambda_{2}\lambda_{3}\eta_{3}(X)\eta_{2}(Y)-\lambda_{2}\lambda_{3}\eta_{2}(X)\eta_{3}(Y)\\
&=g(X,Y)-\eta_{\lambda}(X)\eta_{\lambda}(Y).
\end{align*}
Now we pass to prove the second part of the theorem. By a very long
but straightforward computation one can prove the relation
\begin{align*}
N^{(1)}_{\phi_{\lambda}}=N^{(1)}_{\phi_1}+N^{(1)}_{\phi_2}+N^{(1)}_{\phi_3} + \lambda_{1}\lambda_{2}N^{(1)}_{1,2} + \lambda_{1}\lambda_{3}N^{(1)}_{1,3} +
\lambda_{2}\lambda_{3}N^{(1)}_{2,3}.
\end{align*}
Since $(\phi_\alpha,\xi_\alpha,\eta_\alpha)_{_{\alpha\in\left\{1,2,3\right\}}}$ is a hyper-normal
almost contact 3-structure, we have that
$N^{(1)}_{\phi_1}=N^{(1)}_{\phi_2}=N^{(1)}_{\phi_3}=0$ and, by
Theorem \ref{concomitant},  $N^{(1)}_{\alpha,\beta}=0$ for any
$\alpha,\beta\in\left\{1,2,3\right\}$. Therefore $N^{(1)}_{\phi_{\lambda}}=0$.
\end{proof}

\begin{corollary}\label{sphere3s}
Let $(\phi_\alpha,\xi_\alpha,\eta_\alpha,g)$ be an almost contact metric $3$-structure on a $(4n+3)$-dimensional smooth manifold $M$. Then
$\{(\eta_{1},\Omega_{1}),(\eta_{2},\Omega_{2}),(\eta_{3},\Omega_{3})\}$, being $\Omega_\alpha:=g(\cdot\,,\phi_{\alpha}\cdot)$, is an almost cosymplectic
sphere which is both round and taut.
\end{corollary}
\begin{proof}
With the notation of Theorem \ref{normal}, for each $\alpha\in\left\{1,2,3\right\}$ let $\Omega_{\alpha}:=g(\cdot\, ,\phi_\alpha \cdot)$ be the fundamental $2$-form of the almost contact metric structure $(\phi_\alpha,\xi_\alpha,\eta_\alpha)$.
For each $\lambda=(\lambda_{1},\lambda_{2},\lambda_{3})\in\mathbb{S}^2$ let $(\phi_\lambda,\xi_\lambda,\eta_\lambda,g)$ denote the almost contact metric structure defined by \eqref{acs}, and let $\Omega_\lambda:=g(\cdot\, ,\phi_\lambda \cdot)$ be the corresponding fundamental $2$-form.
Then by the general theory of almost contact metric structures, one has that $\eta_{\lambda}\wedge\Omega_{\lambda}^{2n+1}\neq 0$. This proves that the family
$\left\{\left(\eta_{\lambda},\Omega_{\lambda}\right)\right\}_{\lambda\in\mathbb{S}^2}$ defines an almost cosymplectic sphere on $M$.  Directly from Theorem \ref{normal} it follows that such an almost cosymplectic sphere is round. Thus it
only remains to prove that it is taut. To this aim, let us consider a $\phi$-basis, i.e. a (local) orthonormal basis $\left(\xi_{1},\xi_{2},\xi_{3},X_{1},\ldots,X_{n},Y_{1},\ldots,Y_{n},Z_{1},\ldots,Z_{n},U_{1},\ldots,U_{n}\right)$ such that $X_i\in\Gamma(\mathcal H)$, $Y_{i}=\phi_{1}X_i$, $Z_{i}=\phi_{2}X_i$, $U_{i}=\phi_{3}X_i$ for all $i\in\left\{1,\ldots,n\right\}$.
Using that $\Omega_\lambda=\lambda_{1}\Omega_1 + \lambda_{2}\Omega_2 + \lambda_{3}\Omega_3$, one easily finds that
\begin{equation}\label{conti3}
\Omega_{\lambda}(\xi_{1},\xi_{2})=-\lambda_{3}, \ \ \ \Omega_{\lambda}(\xi_{1},\xi_{3})=\lambda_{2}, \ \ \ \Omega_{\lambda}(\xi_{2},\xi_{3})=-\lambda_{1}
\end{equation}
and
\begin{equation}\label{conti4}
\Omega_{\lambda}(\xi_{\alpha},X_{i})=\Omega_{\lambda}(\xi_{\alpha},Y_{i})=\Omega_{\lambda}(\xi_{\alpha},Z_{i})=\Omega_{\lambda}(\xi_{\alpha},U_{i})=0
\end{equation}
for any $\alpha\in\left\{1,2,3\right\}$ and $i\in\left\{1,\ldots,n\right\}$. By using  \eqref{conti3} and \eqref{conti4}, one has that
\begin{equation}\label{eq-tocompute}
\begin{aligned}
&\eta_\lambda\wedge\Omega_\lambda^{2n+1}\left(\xi_{1},\xi_{2},\xi_{3},X_{1},\ldots,X_{n},Y_{1},\ldots,Y_{n},Z_{1},\ldots,Z_{n},U_{1},\ldots,U_{n}\right)\\
&=(2n+1)(\eta_\lambda(\xi_1)\Omega_\lambda(\xi_2,\xi_3)\Omega_\lambda^{2n}\left(X_{1},\ldots,X_{n},Y_{1},\ldots,Y_{n},Z_{1},\ldots,Z_{n},U_{1},\ldots,U_{n}\right)\\
&\quad-\eta_\lambda(\xi_2)\Omega_\lambda(\xi_1,\xi_3)\Omega_\lambda^{2n}\left(X_{1},\ldots,X_{n},Y_{1},\ldots,Y_{n},Z_{1},\ldots,Z_{n},U_{1},\ldots,U_{n}\right)\\
&\quad+\eta_{\lambda}(\xi_3)\Omega_\lambda(\xi_1,\xi_2)\Omega_\lambda^{2n}\left(X_{1},\ldots,X_{n},Y_{1},\ldots,Y_{n},Z_{1},\ldots,Z_{n},U_{1},\ldots,U_{n}\right))\\
&=(2n+1)\bigl( \eta_{\lambda}(\xi_1)\Omega_\lambda(\xi_2,\xi_3)-\eta_\lambda(\xi_2)\Omega_\lambda(\xi_1,\xi_3)+\eta_\lambda(\xi_3)\Omega_\lambda (\xi_{1},\xi_{2}) \bigr)\cdot\\
&\quad \Omega_\lambda^{2n}\left(X_{1},\ldots,X_{n},Y_{1},\ldots,Y_{n},Z_{1},\ldots,Z_{n},U_{1},\ldots,U_{n}\right).
\end{aligned}
\end{equation}
In order to compute the last term, first notice that
\begin{align*}
\Omega_\lambda^{2n}=(\lambda_{1}\Omega_1 + \lambda_{2}\Omega_2 + \lambda_{3}\Omega_3)^{2n}=\sum_{p_{1}+p_{2}+p_{3}=2n}\frac{(2n)!}{p_{1}!p_{2}!p_{3}!}\lambda_{1}^{p_{1}}\lambda_{2}^{p_{2}}\lambda_{3}^{p_{3}}\Omega_1^{p_{1}}\wedge\Omega_2^{p_{2}}\wedge\Omega_3^{p_{3}}.
\end{align*}
Using the fact that
\begin{equation}\label{eq1}
\begin{aligned}
\Omega_1&(X_{i},Y_{j})=\Omega_1(Z_{i},U_{j})=\Omega_2(X_{i},Z_{j})\\
&=-\Omega_2(Y_{i},U_{j})=\Omega_3(X_{i},U_{j})=\Omega_3(Y_{i},Z_{j})=-\delta_{ij},
\end{aligned}
\end{equation}
\begin{align}\label{eq2}
\Omega_\alpha(X_{i},X_{j})=\Omega_\alpha(Y_{i},Y_{j})=\Omega_\alpha(Z_{i},Z_{j})=\Omega_\alpha(U_{i},U_{j})=0,\qquad \alpha=1,2,3,
\end{align}
and that $\Omega_1, \Omega_2, \Omega_3$ on the other pairs of basis elements vanish as well, one obtains
\begin{align*}
&\Omega_\lambda^{2n}\left(X_{1},\ldots,X_{n},Y_{1},\ldots,Y_{n},Z_{1},\ldots,Z_{n},U_{1},\ldots,U_{n}\right)\\
&=\sum_{q_{1}+q_{2}+q_{3}=n}\frac{(2n)!}{(2q_{1})!(2q_{2})!(2q_{3})!}\lambda_{1}^{2q_{1}}\lambda_{2}^{2q_{2}}\lambda_{3}^{2q_{3}}\cdot\\
&\quad\qquad\Omega_1^{2q_{1}}\wedge\Omega_2^{2q_{2}}\wedge\Omega_3^{2q_{3}}\left(X_{1},\ldots,X_{n},Y_{1},\ldots,Y_{n},Z_{1},\ldots,Z_{n},U_{1},\ldots,U_{n}\right)\\
&=\sum_{q_{1}+q_{2}+q_{3}=n}\frac{(2n)!}{(2q_{1})!(2q_{2})!(2q_{3})!}\lambda_{1}^{2q_{1}}\lambda_{2}^{2q_{2}}\lambda_{3}^{2q_{3}}\cdot\frac{1}{2^{2n}}\beta
\end{align*}
where, owing to \eqref{eq1} and  \eqref{eq2}, $\beta$ coincides with the number of permutations of $X_{1},\ldots,X_{n},Y_{1},\ldots,Y_{n},Z_{1},\ldots,Z_{n},U_{1},\ldots,U_{n}$ in which the tensorial product \allowbreak
$\Omega_1^{\otimes 2q_{1}}\otimes\Omega_2^{\otimes 2q_{2}}\otimes\Omega_3^{\otimes 2q_{3}}$ does not vanish. Precisely,
\begin{align*}
\beta=\binom{n}{q_{1},q_{2},q_{3}}\cdot(2q_{1})!\cdot(2q_{2})!\cdot(2q_{3})!\cdot 2^{2n}.
\end{align*}
Therefore
\begin{equation}\label{eq-omega}
\begin{aligned}
&\Omega_\lambda^{2n}\left(X_{1},\ldots,X_{n},Y_{1},\ldots,Y_{n},Z_{1},\ldots,Z_{n},U_{1},\ldots,U_{n}\right)\\
&=(2n)!\sum_{q_{1}+q_{2}+q_{3}=n} \binom{n}{q_{1},q_{2},q_{3}}
\lambda_{1}^{2q_{1}}\lambda_{2}^{2q_{2}}\lambda_{3}^{2q_{3}}\\
&=(2n)!(\lambda_{1}^{2}+\lambda_{2}^{2}+\lambda_{3}^{2})^{n}\\
&=(2n)!
\end{aligned}
\end{equation}
By using  \eqref{eq-omega}, it follows from \eqref{eq-tocompute}  that the almost cosymplectic sphere is taut.
\end{proof}

We now illustrate an example of almost contact $3$-structure on a $7$-dimensional manifold such that the $1$-forms $\eta_{1}$, $\eta_{2}$, $\eta_{3}$ have different constant ranks.

\begin{example}
Let $\mathfrak{g}$ be the $7$-dimensional Lie algebra with basis $\{X_{1},X_{2},X_{3},X_{4},$ $\xi_1,\xi_2,\xi_3\}$ and non-zero Lie brackets
\begin{equation*}
[X_{1},X_{4}]=\xi_3, \ \ \ [\xi_1,\xi_2]=\xi_3.
\end{equation*}
Let $G$ be a Lie group whose Lie algebra is $\mathfrak{g}$. We define an almost contact $3$-structure $(\phi_\alpha,\xi_\alpha,\eta_\alpha)$, $\alpha\in\left\{1,2,3\right\}$, on $G$ by setting
\begin{equation*}
\eta_{\alpha}(X_{i})=0, \ \ \ \eta_{\alpha}(\xi_\beta)=\delta_{\alpha\beta}
\end{equation*}
for each $i\in\left\{1,\ldots,4\right\}$ and $\alpha,\beta\in\left\{1,2,3\right\}$, and by defining
\begin{gather*}
\phi_{1}X_{1}=X_{2}, \ \ \ \phi_{1}X_{2}=-X_{1}, \ \ \ \phi_{1}X_{3}=X_{4} , \ \ \ \phi_{1}X_{4}=-X_{3}, \\
\phi_{2}X_{1}=X_{3}, \ \ \ \phi_{2}X_{2}=-X_{4}, \ \ \ \phi_{2}X_{3}=-X_{1} , \ \ \ \phi_{2}X_{4}=X_{2}, \\
\phi_{3}X_{1}=X_{4}, \ \ \ \phi_{3}X_{2}=X_{3}, \ \ \ \phi_{3}X_{3}=-X_{2} , \ \ \ \phi_{3}X_{4}=-X_{1},\\
\phi_{\alpha}\xi_{\beta} = \epsilon_{\alpha\beta\gamma}\xi_\gamma,
\end{gather*}
where $\epsilon_{\alpha\beta\gamma}$ is the total antisymmetric symbol. From the definition, it follows that $d\eta_{1}=d\eta_{2}=0$. Next, we have $2d\eta_{3}(X_{1},X_{4})=X_{1}(\eta_{3}(X_{4}))-X_{4}(\eta_{3}(X_{1}))-\eta_{3}([X_{1},X_{4}])=-\eta_{3}([X_{1},X_{4}])=-\eta_{3}(\xi_{3})=-1$, and $2d\eta_{3}(\xi_{1},\xi_{2})=-1$ by a similar computation.
One also checks that $d \eta_{3}$ is zero on any other pair of basis vector fields. Thus
$(\eta_{3}\wedge (d \eta_{3})^{2})(\xi_{3},\xi_{1},\xi_{2},X_{1},X_{4})$ is a non-zero constant at every point and $\eta_{3}\wedge (d \eta_{3})^{3}\equiv 0$.
Therefore  $\eta_1$ and $\eta_2$ have rank $1$, while $\eta_3$ has rank $5$.
\end{example}

The above example motivates the following definition.

\begin{definition}\label{quasicontact}
A \emph{quasi-contact} $p$-sphere of \emph{rank} $k$ is an almost cosymplectic $p$-sphere $\left\{(\eta_{\lambda},\Omega_{\lambda})\right\}_{\lambda\in\mathbb{S}^p}$ such that $\d\Omega_{\lambda}=0$ for any $\lambda\in\mathbb{S}^p$ and there exists a positive integer
 $k$ such that  all $1$-forms $\eta_\lambda$ have the same constant rank $k$, i.e. each  $1$-form $\eta_\lambda$ has Cartan class $k$,  for any $\lambda\in\mathbb{S}^p$.
\end{definition}

Of course contact $p$-spheres are examples of almost cosymplectic $p$-spheres which  are quasi-contact, having the maximal possible rank $2n+1=\dim(M)$. Further, notice that any cosymplectic $p$-sphere is quasi-contact of minimal rank $1$, being $\eta_\lambda$ closed and
nonsingular. Now we shall present a class of examples of  almost cosymplectic $p$-spheres which are quasi-contact of rank
$1<k<\dim(M)$.

\begin{theorem}
With  the notation of Theorem \ref{normal}, if $(M,\phi_\alpha,\xi_\alpha,\eta_\alpha,g)_{\alpha\in\left\{1,2,3\right\}}$ is a $3$-quasi-Sasakian manifold
of rank $4l+3$ then, for every $\lambda\in\bS^{2}$, $(\phi_{\lambda},\xi_{\lambda},\eta_{\lambda},g)$ defines a quasi-Sasakian structure of the same rank. In particular, the almost
cosymplectic sphere generated by $\{(\eta_{1},\Omega_{1}),(\eta_{2},\Omega_{2}),(\eta_{3},\Omega_{3})\}$   is quasi-contact of rank $4l+3$.
\end{theorem}
\begin{proof}
By Theorem \ref{normal} we already know that $(\phi_\lambda,\xi_\lambda,\eta_\lambda,g)$ is
a normal almost contact metric structure on $M$. We have to prove
that the fundamental $2$-form $\Omega_\lambda$ is closed. However, notice that
$\Omega_\lambda=\lambda_{1}\Omega_{1}+\lambda_{2}\Omega_{2}+\lambda_{3}\Omega_{3}$.
Since each structure $(\phi_\alpha,\xi_\alpha,\eta_\alpha,g)$ is
quasi-Sasakian, we get that $d\Omega_\lambda=0$.
We pass to prove that the rank of $\eta_\lambda$ is the same of that of the 3-quasi-Sasakian manifold
$(M,\phi_\alpha,\xi_\alpha,\eta_\alpha,g)_{\alpha\in\left\{1,2,3\right\}}$, namely $2p+1=4l+3$. We will prove this by showing that ${\mathcal E}  ={\mathcal C} $, where the distributions ${\mathcal E}$  and  ${\mathcal C}$ are defined by \eqref{eq:C} and \eqref{eq:E}, respectively. First we prove that ${\mathcal E}  \subset{\mathcal C} $.
Notice that ${\mathcal E}  \subset {\mathcal H}\subset \ker(\eta_{\lambda})$. Consequently, since, on ${\mathcal E}  $, $d\eta_\alpha=0$ for any
$\alpha\in\left\{1,2,3\right\}$, we get that $d\eta_\lambda =0$ on ${\mathcal E}  $. Hence ${\mathcal E}  \subset{\mathcal C} $. In order to prove the
other inclusion, we  observe that the endomorphism of $\mathcal H$ defined by $\rho  = -\lambda_{2}I + \lambda_{3}\phi_{1} +
\lambda_{1}\phi_{3}$ is in fact an isomorphism. Indeed a straightforward computation shows that $(\phi_{2}\rho)^{2} = -I$, hence the mapping $\phi_2 \rho$
is an isomorphism. It follows that $\rho$ is an isomorphism too.
Now let us consider an arbitrary $X\in\Gamma({\mathcal C} )$ and
let us prove that $X\in\Gamma({\mathcal E})$.  We
can decompose $X$ in its components along the horizontal and the
Reeb distribution, as follows
\begin{equation*}
X = X_{\mathcal H} + f_1 \xi_1 + f_2 \xi_2 + f_3 \xi_3
\end{equation*}
where $f_\alpha=\eta_\alpha(X)$ for each
$\alpha\in\left\{1,2,3\right\}$. The condition $\eta_\lambda(X)=0$ yields
\begin{equation}\label{first}
\lambda_{1}f_1 + \lambda_{2}f_2  + \lambda_{3}f_3  =0,
\end{equation}
whereas the conditions
$d\eta_\lambda(\xi_1,X)=d\eta_\lambda(\xi_2,X)=d\eta_\lambda(\xi_3,X)=0$ yield
\begin{align}
\lambda_{2}f_{3}-\lambda_{3}f_{2} &= 0,  &
-\lambda_{1}f_{3}+\lambda_{3}f_{1} &= 0,  &
\lambda_{1}f_{2}-\lambda_{2}f_{1} &= 0. \label{second}
\end{align}
It is easy to see that the homogeneous linear system of the four
equations \eqref{first}, \eqref{second} admits
non-zero solutions if and only if
$\lambda_{1}=\lambda_{2}=\lambda_{3}=0$. Since this circumstance can
not happen, we conclude that $f_{1}=f_{2}=f_{3}=0$ and hence
$X\in{\mathcal H}$. Finally, in order to prove that $X$ belongs to
${\mathcal E}  $, it remains to prove that
$d\eta_{\alpha}(X,Y)=0$ for any $Y\in\Gamma(TM)$ and for any
$\alpha\in\left\{1,2,3\right\}$. In fact, in view of  \cite[Lemma
5.4]{agag08} it is enough to prove that $d\eta_{\alpha}(X,Y)=0$  for
\emph{some} $\alpha\in\left\{1,2,3\right\}$, for instance for
$\alpha=3$. By  \cite[Corollary 3.8]{agag08} we have that
$d\eta_{3}(X,\xi_\beta)=0$ for any $\beta\in\left\{1,2,3\right\}$.
Thus we may assume that $Y\in\Gamma({\mathcal H})$. Since $\rho$ is
an isomorphism, there exists $Y'\in\Gamma({\mathcal H})$ such that
$Y=\rho Y' =
-\lambda_{2}Y'+\lambda_{3}\phi_{1}Y'+\lambda_{1}\phi_{3}Y'$.
Therefore, by using Lemma 5.2 -- Lemma 5.3 of \cite{agag08}, we get
\begin{align*}
d\eta_{3}(X,Y)&=d\eta_{3}(X,\rho Y')=\lambda_{1}d\eta_{3}(X,\phi_{3}Y') - \lambda_{2}d\eta_{3}(X,Y') + \lambda_{3}d\eta_{3}(X,\phi_{1}Y')\\
&=\lambda_{1}d\eta_{1}(X,\phi_{1}Y') + \lambda_{2}d\eta_{2}(X,\phi_{1}Y') + \lambda_{3}d\eta_{3}(X,\phi_{1}Y')
\\&=d\eta_\lambda(X,\phi_{1}Y')=0,
\end{align*}
where the last equality follows from the fact that
$X\in\Gamma({\mathcal C} )$.
\end{proof}

\begin{corollary}\label{sasakian}
Any $3$-Sasakian manifold admits an almost cosymplectic sphere of
Sasakian structures which is both taut and round.
\end{corollary}

\begin{corollary}
Any $3$-cosymplectic manifold admits an almost cosymplectic sphere of
cosymplectic structures which is both taut and round.
\end{corollary}

\begin{remark}
Notice that Corollary \ref{sasakian} improves the result of Zessin
that  any 3-Sasakian manifold admits a contact sphere which is both taut and round (cf.
\cite[Proposition 4]{zessin05}).
\end{remark}

\medskip

Arguing in a similar way to the proof of Theorem \ref{normal} we can find the following classes of examples of cosymplectic spheres.

\begin{example}
Let $N$ be a smooth manifold endowed with a hyperholomorphic symplectic structure in the sense of \cite{bande} (note that these geometric structures were
also studied in \cite{xu}, but with a different name). Namely, on $N$ there are defined three symplectic structures $\omega_1$, $\omega_2$, $\omega_3$
related to each other by means of the relations
\begin{equation*}
\omega_{\alpha}^{\sharp}\circ\omega_{\beta}^{\flat}=-\omega_{\beta}^{\sharp}\circ\omega_{\alpha}^{\flat}
\end{equation*}
for any $\alpha,\beta\in\left\{1,2,3\right\}$, $\alpha\neq\beta$. As a consequence, one can define three almost complex structures which satisfy the quaternionic identities.
On $M:=N\times\mathbb{R}^3$ we define for each even permutation $(\alpha,\beta,\gamma)$
of $\left\{1,2,3\right\}$
\begin{equation*}
\eta_\alpha:=dt_\alpha, \ \ \ \Omega_\alpha:=\omega_\alpha+\eta_{\beta}\wedge\eta_{\gamma}
\end{equation*}
where $(t_1,t_2,t_3)$ are the global coordinates of $\mathbb{R}^3$. Then, arguing as in the proof of Corollary \ref{sphere3s}, one can prove that
$\left\{(\eta_{1},\Omega_{1}),(\eta_{2},\Omega_{2}),(\eta_{3},\Omega_{3})\right\}$ generates a cosymplectic sphere on $M$.
\end{example}

\end{document}